\documentclass[12pt]{amsart}
\usepackage[colorlinks=true,citecolor=magenta,linkcolor=magenta]{hyperref}
\usepackage{amssymb,verbatim,amscd,amsmath,graphicx}
\usepackage{enumerate}
\usepackage{graphicx}
\usepackage{caption}
\usepackage{subcaption}
\usepackage{pdfsync}
\usepackage{color,colortbl}
\usepackage{fullpage}
\usepackage{bbm}
\usepackage{comment}

\usepackage{pgf,tikz,pgfplots}
\usepackage{mathrsfs}
\usetikzlibrary{arrows}

\numberwithin{equation}{section}
\newtheorem{theorem}{Theorem}[section]
\newtheorem{lemma}[theorem]{Lemma}

\newtheorem{corollary}[theorem]{Corollary}

\theoremstyle{definition}
\newtheorem{definition}[theorem]{Definition}
\newtheorem{construction}[theorem]{Construction}

\newtheorem{def-prop}[theorem]{Definition-Proposition}

\newtheorem{example}[theorem]{Example}

\newtheorem*{acknowledgement}{Acknowledgements}

\newtheorem{question}[theorem]{Question}

\DeclareMathOperator{\Ass}{Ass}

\DeclareMathOperator{\link}{link}
\DeclareMathOperator{\del}{del}

\DeclareMathOperator{\gr}{gr}

\DeclareMathOperator{\pol}{pol}

\newcommand{\ZZ}{{\mathbb Z}}
\newcommand{\NN}{{\mathbb N}}

\newcommand{\kk}{{\mathbbm k}}

\newcommand{\cupdot}{\mathbin{\mathaccent\cdot\cup}}

\def\K{{\mathcal K}}

\def\mm{{\frak m}}

\def\pp{{\frak p}}

\def\1{{\bf 1}}
\def\0{{\bf 0}}

\begin{document}
	
\title{Symbolic powers of cover ideals of graphs and Koszul property}

\author{Yan Gu}
\address{Soochow University, Soochow College \& School of Mathematical Sciences, Suzhou 215006, P.R. China.}
\email{guyan@suda.edu.cn}
%\urladdr{any website?}

\author{Huy T\`ai H\`a}
\address{Tulane University \\ Department of Mathematics \\
	6823 St. Charles Ave. \\ New Orleans, LA 70118, USA}
\email{tha@tulane.edu}
%\urladdr{???}

\author{Joseph W. Skelton}
\address{Tulane University \\ Department of Mathematics \\
	6823 St. Charles Ave. \\ New Orleans, LA 70118, USA}
\email{jskelton@tulane.edu}
%\urladdr{any website?}


\begin{abstract}
We show that attaching a whisker (or a pendant) at the vertices of a cycle cover of a graph results in a new graph with the following property: all symbolic powers of its cover ideal are Koszul or, equivalently, componentwise linear. This extends previous work where the whiskers were added to all vertices or to the vertices of a vertex cover of the graph.
\end{abstract}

\maketitle

%%%%%%%%%%%%%%%%%%%%%%%%%%%%%%%%%%%%%%%%%%%%%%%%%%%%

\section{Introduction} \label{sec.intro}

In this paper, we examine how simple modifications of a graph may result in nice algebraic properties of the corresponding monomial ideals. Specifically, given an arbitrary (simple and undirected) graph $G$ we present a sufficient condition on a subset $S$ of the vertices such that the graph $G \cup W(S)$ obtained from $G$ by attaching a \emph{whisker} (also referred to as \emph{pendant}) to each vertex in $S$ has the following property: all \emph{symbolic} powers of its \emph{cover ideal} are \emph{Koszul ideals}. (By attaching a whisker at a vertex $x$ in a graph $G$ we add a new vertex $y$ and the edge $\{x,y\}$ to $G$.)

Both aspects in our work, Koszul ideals, or equivalently componentwise linear ideals, and combinatorially modifying graphs for prescribed algebraic properties have long and rich histories. Koszul algebras were first introduced by Priddy \cite{Priddy1970} and later generalized to Koszul modules by Herzog and Iyengar \cite{HS2005}. This property was also considered by \c{S}ega  \cite{Sega2001} and Herzog \cite{Herzog2003}. A finitely generated $R$-module $M$, for a polynomial ring $R$ over a field $\kk$, is called a \emph{Koszul module} if its associated graded module $\gr_\mm(M) = \bigoplus_{i \ge 0} \mm^iM/\mm^{i+1}M$, with respect to the homogeneous maximal ideal $\mm$ of $R$, has a \emph{linear resolution} over the associated graded ring $\gr_\mm(R) = \bigoplus_{i \ge 0} \mm^i/\mm^{i+1}$. Koszul algebras and Koszul modules have been investigated extensively and found significant applications in many different areas of mathematics. It has been an important problem to find new classes of Koszul algebras and Koszul modules (see \cite{C2009, CC2013, CHTV1997, DAli2017, EHH2014, Nguyen2014, NTV2015} for a small sample of recent work on this problem). Our interest in this paper is on monomial ideals which are Koszul modules.

A closely related property to being Koszul for ideals is that of being Golod --- an ideal $J \subseteq R$ is called \emph{Golod} if the Koszul complex of $R/J$ with respect to the variables in $R$ admits trivial \emph{Massey operations} (see \cite{Golod1962}) or, equivalently, if the Poincar\'e series of $R/J$ coincides with the upper bound given by Serre (cf. \cite{Serre1965}). Herzog and Huneke \cite[Theorems 1.1 and 2.3]{HH2013} showed that if $\text{char } \kk = 0$ and $J$ is a homogeneous ideal in $R$ then $J^k$ and $J^{(k)}$ are Golod for all $k \ge 2$. Furthermore, Ahangari Maleki \cite[Theorem 4.7]{Maleki2019} proved that if $J$ is a proper Koszul ideal in $R$ then $J$ is Golod. To view these results under a unified perspective the following problem arises naturally: \emph{find polynomial ideals whose powers and/or symbolic powers are all Koszul}.

On the other hand, in an attempt to understand Cohen-Macaulay monomial ideals, a result of Villarreal \cite{Villarreal1990} showed that for any graph $G = (V_G,E_G)$, the graph $G \cup W(V_G)$ obtained by adding a whisker at every vertex is a \emph{Cohen-Macaulay} graph. Equivalent, the cover ideal $J(G \cup W(V_G))$ has a \emph{linear resolution}. In fact, it was later shown (see \cite{DE2009, Woodroofe2009}) that $G \cup W(V_G)$ is \emph{vertex decomposable} and \emph{shellable}. Note that, in general,
$$\text{vertex decomposable } \Rightarrow \text{ shellable } \Rightarrow \text{ sequentially Cohen-Macaulay}.$$
These studies were generalized to whiskering and partial whiskering of simplicial complexes (see \cite{BVT2013, BFHVT2015, CN2012, FH2008}). For instance, in \cite{FH2008, BFHVT2015}, necessary and sufficient conditions were given on a subset $S$ of the vertices of a simplicial complex such that adding a whisker at each vertex in $S$ results in a sequentially Cohen-Macaulay or vertex decomposable simplicial complex. These investigations lie under the same umbrella of the following question: \emph{how to combinatorially modify a graph or a simplicial complex to obtain prescribed algebraic properties of the corresponding monomial ideals}?

The aforementioned problem and question are deeply connected by the following facts: (a) a homogeneous ideal in $R$ is Koszul if and only if it is \emph{componentwise linear} (see \cite[Theorem 3.2.8]{Romer2001} and \cite[Proposition 4.9]{Ya2000}); and (b) the cover ideal of a graph is componentwise linear if and only if the graph is sequentially Cohen-Macaulay (see \cite[Theorem 2.1]{HH1999} and \cite[Corollary 4.11]{Ya2000}). Particularly, if a graph is vertex decomposable then its cover ideal is Koszul.

This connection was further illustrated in a recent work of Dung, Hien, Nguyen and Trung \cite{DHNT2019}, where it was shown that for any graph $G = (V_G,E_G)$, all symbolic powers of the cover ideal of $G \cup W(V_G)$ are Koszul. Selvaraja \cite{Selvaraja2019} then showed that one only needs to whisker at the vertices of a \emph{vertex cover} of $G$ to get a new graph for which all symbolic powers of the cover ideal are Koszul. 

By a \emph{cycle cover} of a graph $G$, we mean a subset $S$ of the vertices in $G$ such that every cycle (i.e., closed path) in $G$ contains at least a vertex in $S$. Obviously any vertex cover is a cycle cover. However, a cycle cover may have much less vertices than a vertex cover, see the following example.

\begin{example}
	\label{ex.cycleCov}
	In the graph depicted below any minimum vertex cover has 3 vertices. On the other hand, both $\{x_2\}$ and $\{x_4\}$ are cycle covers.
\begin{figure}[h]
	\centering
	\begin{tikzpicture}[scale=2.5,line cap=round,line join=round,>=triangle 45,x=1cm,y=1cm]
		\clip(0.49852191097697623,0.20073201995513232) rectangle (3.449839663356829,1.7798786614882518);
		\draw [line width=1pt] (1,0.5)-- (1,1.5);
		\draw [line width=1pt] (1,1.5)-- (2,1.5);
		\draw [line width=1pt] (2,1.5)-- (2,0.5);
		\draw [line width=1pt] (2,0.5)-- (1,0.5);
		\draw [line width=1pt] (1,0.5)-- (2,1.5);
		\draw [line width=1pt] (2,0.5)-- (3,0.5);
		\begin{scriptsize}
		\draw [fill=black] (1,0.5) circle (1pt);
		\draw[color=black] (0.90061332951202777,0.5466494578152392) node {$x_2$};
		\draw [fill=black] (1,1.5) circle (1pt);
		\draw[color=black] (0.90061332951202777,1.5466311894660747) node {$x_1$};
		\draw [fill=black] (2,1.5) circle (1pt);
		\draw[color=black] (2.1001753495416535,1.5466311894660747) node {$x_4$};
		\draw [fill=black] (2,0.5) circle (1pt);
		\draw[color=black] (2.1001753495416535,0.5466494578152392) node {$x_3$};
		\draw [fill=black] (3,0.5) circle (1pt);
		\draw[color=black] (3.1205570811924867,0.5566311894660747) node {$x_5$};
		\end{scriptsize}
	\end{tikzpicture}
		\caption{cycle covers may have much less vertices than vertex covers.}
\end{figure}
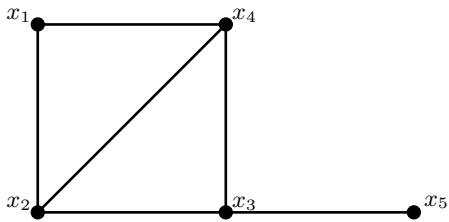
\end{example}

In this paper, we generalize the result of Selvaraja \cite{Selvaraja2019} a step further, proving that the same conclusion holds true after adding whiskers at the vertices of a cycle cover of the graph. Our main theorem is stated as follows.

\medskip

\noindent{\bf Theorem \ref{thm.main}.} Let $G$ be a graph and let $S$ be a cycle cover of $G$. Let $H$ be a graph obtained by adding at least one whisker at each vertex in $S$ and let $J$ be its cover ideal. Then, $J^{(k)}$ is Koszul for all $k \ge 1$.

\medskip

To prove Theorem \ref{thm.main}, we view the \emph{polarization} of $J^{(k)}$ as the cover ideal of a new graph $H_k$ constructed from $H$ by duplicating its vertices; this construction was introduced by Seyed Fakhari in \cite{SF2018}. This approach was also used by Selvaraja \cite{Selvaraja2019} in combination with an induction on $k$ and the size of $G$. The base case of \cite{Selvaraja2019}, after removing the vertices in a vertex cover of $G$, reduces to a graph consisting of only isolated vertices (i.e., having no edges). This method, unfortunately, does not work when $S$ is a cycle cover of $G$; removing $S$ from $G$ gives a forest. To overcome this obstruction, we further employ a recent construction of Kumar and Kumar \cite{KK2020} to consider graphs $H(k_1, \dots, k_m)$ obtained by duplicating the edges in $H$; here, $m$ is the number of edges in $H$ and $(k_1, \dots, k_m) \in \NN^m$. Note that $H_k = H(k, \dots, k).$ Theorem \ref{thm.main} is a consequence of Theorem \ref{thm.Gedge}, where it is shown that for suitable tuples $(k_1, \dots, k_m) \in \NN^m$, the graphs $H(k_1, \dots, k_m)$ are vertex decomposable.

Selvaraja \cite{Selvaraja2019} considered attaching more than just whiskers at the vertices in $G$. Specifically, whiskers were replaced by \emph{pure} and \emph{non-pure star complete graphs}. We also extend Theorem \ref{thm.main} to include the situation where non-pure star complete graphs are attached to the vertices of a cycle cover of the graph; see Theorem \ref{thm.glueStar}. To prove Theorem \ref{thm.glueStar}, we first prove a ``gluing'' result which allows us to glue two graphs with the property that symbolic powers of their cover ideals are Koszul along a leaf to obtain a new graph with the same property; see Theorem \ref{thm.Glue}.

The paper is outlined as follows. In the next section, we collect important notations, terminology and auxiliary results that will be used in the paper. In Section \ref{sec.main}, we prove our main result, Theorem \ref{thm.main}, where whiskers are attached to the vertices of a cycle cover of a graph. In Section \ref{sec.SC}, we extend Theorem \ref{thm.main} to also consider attaching star complete graphs to the vertices of the graph. We prove Theorem \ref{thm.glueStar} in this section.

\begin{acknowledgement}
The first author is supported by the National Natural Science Foundation of China (11501397). The second named author is partially supported by Louisiana Board of Regents (grant \# LEQSF(2017-19)-ENH-TR-25).
\end{acknowledgement}

%%%%%%%%%%%%%%%%%%%%%%%%%%%%%%%%%%%%

\section{Preliminaries} \label{sec.prel}

In this section, we collect basic notation and terminology used in the paper. Throughout the paper, by a \emph{graph} we mean a finite \emph{simple} and \emph{undirected} graph; that is, a finite undirected graph with no loops and no multiple edges. For a graph $G$, we use $V_G$ and $E_G$ to denote the vertex set and the edge set of $G$, respectively. Let $\kk$ be an infinite field. We shall identify the vertices $V_G = \{x_1, \dots, x_n\}$ of $G$ with the variables in the polynomial ring $R = \kk[x_1, \dots, x_n]$.

By adding a \emph{whisker} to a vertex $x$ in $G$ we mean to add a new vertex $y$ to $V_G$ and a new edge $\{x,y\}$ to $E_G$. Let $G \cup W(S)$, for a subset $S \subseteq V_G$, denote the graph obtained from $G$ by adding a whisker at each vertex in $S$.

For a subset $U$ of the vertices in $G$, the \emph{induced subgraph} of $G$ on $U$, denoted by $G_U$, is the graph whose vertex set is $U$ and edge set is $\{\{x,y\} \in E_G ~\big|~ x,y \in U\}.$ We also write $G \setminus U$ for the induced subgraph of $G$ on the complement of $U$. For a subset $F \subseteq E_G$ of the edges in $G$, let $G \setminus F$ be the graph obtained by deleting the edges in $F$ from $G$. For simplicity of notation, we write $G \setminus x$ and $G \setminus e$ for $G \setminus \{x\}$ and $G \setminus \{e\}$, where $x \in V_G$ and $e \in E_G$, respectively.

A \emph{walk} in $G$ is an alternating sequence $x_1, e_1, x_2, e_2, \dots, x_m, e_m, x_{m+1}$ of vertices and edges in $G$, in which $e_i = \{x_i, x_{i+1}\}$. Without creating any confusion, we often write $x_1, x_2, \dots, x_{m+1}$ for such a walk. We also write $xy$ for the edge $\{x,y\}$ in $G$. A \emph{path} is a walk $x_1, \dots, x_{m+1}$, in which the vertices are distinct, except possibly with $x_1 \equiv x_{m+1}$. A \emph{cycle} is a closed path; that is, a path $x_1, \dots, x_{m+1}$ with $x_1 \equiv x_{m+1}$. In this case, we write $x_1, x_2, \dots, x_m$ to denote the cycle.

% graphs, edge, cover ideals

\begin{definition}
	Let $G$ be a graph. Let $W \subseteq V_G$ be a subset of the vertices.
	\begin{enumerate}
		\item $W$ is called an \emph{independent set} if no edge in $G$ connect two vertices in $W$.
		\item $W$ is called a \emph{vertex cover} of $G$ if $G \setminus W$ consists of only isolated vertices or, equivalently, if $V_G \setminus W$ is an independent set.
		\item $W$ is called a \emph{cycle cover} of $G$ if $G \setminus W$ has no cycle.
	\end{enumerate}
\end{definition}

It is easy to see that any subset of an independent set is also an independent set. Thus, the collection of independent sets in $G$ form the faces of a \emph{simplicial complex} over $V_G$. We call this simplicial complex the \emph{independent complex} of $G$, and denote it by $\Delta(G)$.

Let $\Delta$ be a simplicial complex and let $\sigma \in \Delta$ be a face. The \emph{deletion} of $\sigma$ in $\Delta$, denoted by $\del_\Delta(\sigma)$, is the simplicial complex obtained by removing $\sigma$ and all faces containing $\sigma$ from $\Delta$. The \emph{link} of $\sigma$ in $\Delta$, denoted by $\link_\Delta(\sigma)$, is the simplicial complex whose faces are:
$$\{\eta \in \Delta ~\big|~ \eta \cap \sigma = \emptyset, \eta \cup \sigma \in \Delta\}.$$

\begin{definition}
	A simplicial complex $\Delta$ is called \emph{vertex decomposable} if either
	\begin{enumerate}
		\item $\Delta$ is a \emph{simplex}; or
		\item there is a vertex $v$ in $\Delta$ such that
		\begin{enumerate}
			\item all facets of $\del_\Delta(v)$ are facets of $\Delta$ (in this case, $v$ is called a \emph{shedding vertex} of $\Delta$), and
			\item both $\link_\Delta(v)$ and $\del_\Delta(v)$ are vertex decomposable.
		\end{enumerate}
	\end{enumerate}
\end{definition}

A graph $G$ is called \emph{vertex decomposable} if its independent complex $\Delta(G)$ is a vertex decomposable simplicial complex. Shedding vertices of $\Delta(G)$ are also referred to as shedding vertices of $G$.

The \emph{open} and \emph{closed} \emph{neighborhoods} of a vertex $x$ in a graph $G$ are defined to be
$$N_G(x) := \{y \in V_G ~\big|~ \{x,y\} \in E_G\} \text{ and } N_G[x] := N_G(x) \cup \{x\}.$$
A vertex $x$ in $G$ is called a \emph{simplicial vertex} if the induced subgraph of $G$ on $N_G[x]$ is a \emph{complete} graph. It follows from a result of Woodroofe \cite[Corollary 7]{Woodroofe2009} that neighbors of a simplicial vertex in $G$ are shedding vertices of $G$. A direct translation from $\Delta(G)$ to $G$ gives us the following definition for vertex decomposable graphs.

\begin{definition}
	\label{def.vertDec}
	A graph $G$ is said to be \emph{vertex decomposable} if either $G$ consists of isolated vertices or there is a vertex $x$ in $V(G)$ such that
	\begin{enumerate}
		\item no independent set in $G\backslash N_G[x]$ is a maximal independent set in $G\backslash x$, and
		\item $G\backslash x$ and $G\backslash N_G[x]$ are vertex decomposable.
	\end{enumerate}
\end{definition}

We now define the main objects of study in this paper, the \emph{cover ideal} of a graph and their \emph{symbolic powers}.

\begin{definition}
	Let $G$ be a graph and let $R$ be the corresponding polynomial ring. The \emph{cover ideal} $J(G)$ of $G$ is defined as follows:
	$$J(G) := \langle x_{i_1} \cdots x_{i_r} ~\big|~ \{x_{i_1}, \dots, x_{i_r}\} \text{ is a vertex cover of } G\rangle \subseteq R.$$
\end{definition}

The construction of cover ideals gives a one-to-one correspondence between graphs on $n$ vertices and height 2 squarefree monomial ideals in $n$ variables.

\begin{definition}
	Let $I \subseteq R$ be an ideal. For any $k \in \NN$, the $k$-th \emph{symbolic power} of $I$ is defined to be
	$$I^{(k)} := \bigcap_{\pp \in \Ass(R/I)} \left(I^nR_\pp \cap R\right) \subseteq R.$$
\end{definition}

For a squarefree monomial ideal $I$, $I^{(k)}$ is a monomial ideal but in general not squarefree. We will employ a commonly used process, the \emph{polarization}, to obtain a squarefree monomial ideal from $I^{(k)}$. For details of polarization we refer the reader to \cite{HH2011book}.

\begin{definition}\label{pol_def}
	Let $M=x_1^{a_1}\cdots x_n^{a_n}$ be a monomial in  $R=\kk[x_1,\dots,x_n]$. Then we define the squarefree monomial $P(M)$ (the \emph{polarization} of $M$) to be
	$$P(M)=x_{11}\cdots x_{1a_1}x_{21}\cdots x_{2a_2}\cdots x_{n1}\cdots x_{na_n}$$ in the polynomial ring $S=\kk[x_{ij} \mid 1\leq i\leq n,1\leq j\leq a_i]$.
	If $I=(M_1,\dots,M_q)$ is a monomial ideal in $R$, then the \emph{polarization} of $I$, denoted by $I^{\pol}$, is defined as $I^{\pol}=(P(M_1),\dots,P(M_q))$ sitting in an appropriate polynomial ring.
\end{definition}

Polarization is useful in our study, particularly thanks to the following simple lemma.

\begin{lemma}%[\protect{\cite[Lemma 3.5]{SF2018}}]
	\label{lem.SF}
Let $J$ be a monomial ideal such that $J^{\pol}$ is the cover ideal of a graph $G$. If $G$ is vertex decomposable then $J$ is Koszul.
\end{lemma}

\begin{proof} By \cite[Theorem 8.2.5]{HH2011book}, $J^{\pol}$ has linear quotients (see \cite[Section 8.2.1]{HH2011book} for more details on ideals with \emph{linear quotients}). This, together with \cite[Lemma 3.5]{SF2018}, implies that $J$ has linear quotients. The conclusion now follows from \cite[Theorem 8.2.15]{HH2011book}.
\end{proof}

% Needed constructions.

Finally, our inductive argument for the vertex decomposability makes use of the following simple observation of Selvaraja \cite{Selvaraja2019}, which follows directly from the definition.

\begin{lemma}[\protect{\cite[Lemma 3.4]{Selvaraja2019}}]
	\label{lem.Sel4}
	Let $G$ be a graph and suppose that $\{z_1, \dots, z_m\} \subseteq V_G$. Set $\Psi_0 = G$,
	$$\Psi_i = \Psi_{i-1} \setminus z_i \text{ and } \Omega_i = \Psi_{i-1} \setminus N_{\Psi_{i-1}}[z_i] \text{ for } 1 \le i \le m.$$
	Suppose that,
	\begin{enumerate}
		\item[(a)] $z_i$ is a shedding vertex of $\Psi_{i-1}$ for all $1 \le i \le m$,
		\item[(b)] $\Omega_i$ is vertex decomposable for all $1 \le i \le m$, and
		\item[(c)] $\Psi_m$ is vertex decomposable.
	\end{enumerate}
Then, $G$ is vertex decomposable.
\end{lemma}

%\begin{lemma}[\protect{\cite[Lemma 3.5]{Selvaraja2019}}]
%	\label{lem.Sel5}
%	Let $G$ be a vertex decomposable graph and let $x$ be a simplicial vertex in $G$. If $y \in N_G(x)$ then $G \setminus y$ is vertex decomposable.
%\end{lemma}

%%%%%%%%%%%%%%%%%%%%%%%%%%%%%%%%%%%%

\section{Whiskers at Cycle Covers and Koszul Property} \label{sec.main}

In this section, we prove our main result, Theorem \ref{thm.main}, that whiskering at a cycle cover of a graph results in a new graph with the special property that all symbolic powers of its cover ideal are Koszul.

Our work is based on the following constructions introduced by Seyed Fakhari in \cite{SF2018} and by Kumar and Kumar in \cite{KK2020}, which allow us to view the polarization of symbolic powers of the cover ideal of a graph as the cover ideal of new graphs obtained by duplicating vertices and edges of the given graph.

\begin{construction}[Duplicating the vertices]
	\label{con.Gk}
	Let $G$ be a graph over the vertex set $V_G = \{x_1,\dots,x_n\}$ and let $k\geq 1$ be an integer. Construct a new graph $G_k$ whose vertex set and edge set are given as follows:
	\begin{align*}
	V_{G_k} & =\{x_{i,p}\mid 1\leq i\leq n \text{ and }1\leq p\leq k\}, \text{ and } \\
	E_{G_k} & =\big\{\{x_{i,p},x_{j,q}\}\mid \{x_i,x_j\}\in E(G) \text{ and } p+q \leq k+1\big\}.
	\end{align*}
\end{construction}

In Construction \ref{con.Gk}, we call $x_{i,p}$'s the \emph{shadows} of $x_i$. It follows from \cite[Lemma 3.4]{SF2018} that the polarization of $J(G)^{(k)}$ is precisely the cover ideal of $G_k$.

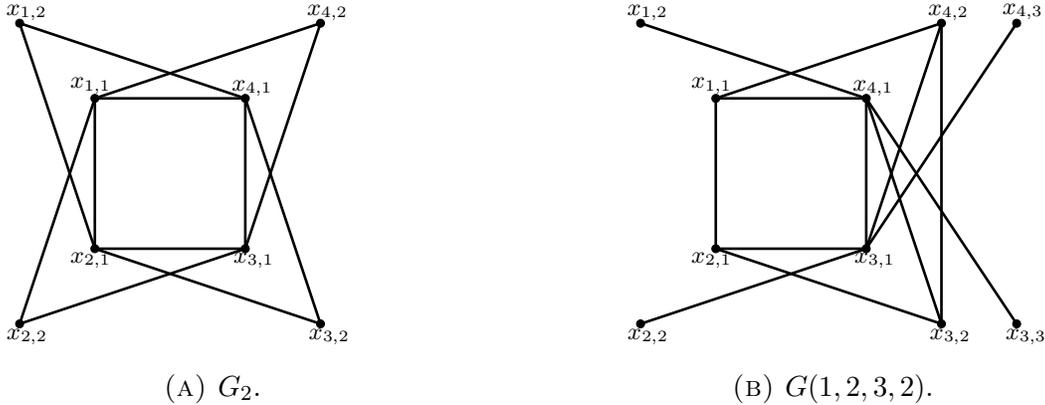
\begin{figure}[h]
	\centering
	\begin{subfigure}[h]{.5\linewidth}
		\centering
		\begin{tikzpicture}[line cap=round,line join=round,>=triangle 45,x=1cm,y=1cm]
		\clip(-0.5633338763149934,-0.420323501532376) rectangle (8.101523972750211,4.4916381232948694);
		\draw [line width=1pt] (2,1)-- (2,3);
		\draw [line width=1pt] (2,3)-- (4,3);
		\draw [line width=1pt] (4,3)-- (4,1);
		\draw [line width=1pt] (4,1)-- (2,1);
		\draw [line width=1pt] (2,3)-- (5,4);
		\draw [line width=1pt] (4,3)-- (1,4);
		\draw [line width=1pt] (4,1)-- (1,0);
		\draw [line width=1pt] (2,1)-- (5,0);
		\draw [line width=1pt] (4,1)-- (5,4);
		\draw [line width=1pt] (4,3)-- (5,0);
		\draw [line width=1pt] (1,0)-- (2,3);
		\draw [line width=1pt] (2,1)-- (1,4);
		\begin{scriptsize}
		\draw [fill=black] (2,1) circle (1.5pt);
		\draw[color=black] (1.9432857157360122,0.8491225317084326) node {$x_{2,1}$};
		\draw [fill=black] (2,3) circle (1.5pt);
		\draw[color=black] (1.9432857157360122,3.1665900561163856) node {$x_{1,1}$};
		\draw [fill=black] (4,3) circle (1.5pt);
		\draw[color=black] (4.092620584789849,3.1384574007622774) node {$x_{4,1}$};
		\draw [fill=black] (4,1) circle (1.5pt);
		\draw[color=black] (4.115126709073135,0.8491225317084326) node {$x_{3,1}$};
		\draw [fill=black] (5,4) circle (1.5pt);
		\draw[color=black] (5.094143115396087,4.13997993136852) node {$x_{4,2}$};
		\draw [fill=black] (1,4) circle (1.5pt);
		\draw[color=black] (1.0936795240419563,4.13997993136852) node {$x_{1,2}$};
		\draw [fill=black] (1,0) circle (1.5pt);
		\draw[color=black] (1.099306055112778,-0.173653061039452833) node {$x_{2,2}$};
		\draw [fill=black] (5,0) circle (1.5pt);
		\draw[color=black] (5.116649239679373,-0.173653061039452833) node {$x_{3,2}$};
		\end{scriptsize}
		\end{tikzpicture}
		\caption{$G_2$.}\label{fig:Gk}
	\end{subfigure}%	
	\begin{subfigure}[h]{.5\linewidth}
		\centering
		\begin{tikzpicture}[line cap=round,line join=round,>=triangle 45,x=1cm,y=1cm]
		\clip(-0.5633338763149934,-0.4203235015323746) rectangle (8.101523972750211,4.4916381232948694);
		\draw [line width=1pt] (2,1)-- (2,3);
		\draw [line width=1pt] (2,3)-- (4,3);
		\draw [line width=1pt] (4,3)-- (4,1);
		\draw [line width=1pt] (4,1)-- (2,1);
		\draw [line width=1pt] (2,3)-- (5,4);
		\draw [line width=1pt] (4,3)-- (1,4);
		\draw [line width=1pt] (4,1)-- (1,0);
		\draw [line width=1pt] (2,1)-- (5,0);
		\draw [line width=1pt] (4,1)-- (5,4);
		\draw [line width=1pt] (4,3)-- (5,0);
		\draw [line width=1pt] (4,1)-- (6,4);
		\draw [line width=1pt] (4,3)-- (6,0);
		\draw [line width=1pt] (5,0)-- (5,4);
		\begin{scriptsize}
		\draw [fill=black] (2,1) circle (1.5pt);
		\draw[color=black] (1.9432857157360122,0.8491225317084326) node {$x_{2,1}$};
		\draw [fill=black] (2,3) circle (1.5pt);
		\draw[color=black] (1.9432857157360122,3.166590056116386) node {$x_{1,1}$};
		\draw [fill=black] (4,3) circle (1.5pt);
		\draw[color=black] (4.092620584789849,3.138457400762278) node {$x_{4,1}$};
		\draw [fill=black] (4,1) circle (1.5pt);
		\draw[color=black] (4.115126709073135,0.8491225317084326) node {$x_{3,1}$};
		\draw [fill=black] (5,4) circle (1.5pt);
		\draw[color=black] (5.094143115396087,4.13997993136852) node {$x_{4,2}$};
		\draw [fill=black] (1,4) circle (1.5pt);
		\draw[color=black] (1.0936795240419563,4.13997993136852) node {$x_{1,2}$};
		\draw [fill=black] (1,0) circle (1.5pt);
		\draw[color=black] (1.099306055112778,-0.173653061039452833) node {$x_{2,2}$};
		\draw [fill=black] (5,0) circle (1.5pt);
		\draw[color=black] (5.116649239679373,-0.173653061039452833) node {$x_{3,2}$};
		\draw [fill=black] (6,4) circle (1.5pt);
		\draw[color=black] (6.0787860527898605,4.156859524580986) node {$x_{4,3}$};
		\draw [fill=black] (6,0) circle (1.5pt);
		\draw[color=black] (6.123798301356433,-0.173653061039452833) node {$x_{3,3}$};
		\end{scriptsize}
		\end{tikzpicture}
		\caption{$G(1,2,3,2)$.}\label{fig:Gedge}
	\end{subfigure}
	\caption{Construction of $G_k$ and $G(k_1, \dots, k_m)$.} \label{fig.duplication}
\end{figure}	

\begin{construction}[Duplicating the edges]
	\label{con.Gedge}
	Let $G$ be a simple graph with vertex set $V_G = \{x_1,\dots, x_n\}$ and edge set $E_G = \{e_1,\dots, e_m\}$.
	\begin{enumerate}
		\item Let $r \in \ZZ_{\ge 0}$ and consider any edge $e = \{x_i, x_j\}$ in $G$. Set
		\begin{align*}
		V(e(r)) & = \{x_{k,p} ~\big|~ k \in \{i,j\} \text{ and } 1 \le p \le r\}, \text{ and } \\
		E(e(r)) & = \{\{x_{i,p}, x_{j,q}\} ~\big|~ p+q \le r+1\}.
		\end{align*}
		\item For an ordered tuple $(k_1, \dots, k_m) \in \ZZ^m_{\ge 0}$, construct a new graph $G(k_1, \dots, k_m)$ whose vertex set and edge set are given by:
		\begin{align*}
		V_{G(k_1, \dots, k_m)} & = \bigcup_{i=1}^m V(e_i(k_i)), \text{ and } \\
		E_{G(k_1, \dots, k_m)} & = \bigcup_{i=1}^m E(e_i(k_i)).
		\end{align*}
	\end{enumerate}
\end{construction}

In Construction \ref{con.Gedge}, we call $x_{i,p}$'s the \emph{shadows} of $x_i$ and edges in $E(e_i(k_i))$ the \emph{shadows} of $e_i$. It is an easy observation that $G_k = G(\underbrace{k, \dots, k}_{m \text{ times}})$. Particularly, for any $k \ge 1$, we have
\begin{align}
\left(J(G)^{(k)}\right)^{\pol} = J(G(\underbrace{k, \dots, k}_{m \text{ times}})). \label{eq.polarization}
\end{align}

\begin{example}
	\label{ex.conDif}
	Let the graph $G=C_4$ with edge set $e_i = \{x_i,x_{i+1}\}$ for $i = 1,2,3$ and $e_4 = \{x_1,x_4\}$. Then the graphs depicted in Figure \ref{fig.duplication} highlight the differences in the Construction \ref{con.Gk} and \ref{con.Gedge}.
\end{example}

The next few lemmas are adaptation of similar statements for vertex-duplication graphs $G_k$ from \cite{Selvaraja2019} to edge-duplication graphs $G(k_1, \dots, k_m)$.

\begin{lemma}[\protect{Compare with \cite[Lemma 3.2]{Selvaraja2019}}]
	\label{lem.Sel2}
	Let $H$ be a graph over the vertex set $V_H=\{x_1,...,x_n\}$ with $m-1$ edges $e_1, \dots, e_{m-1}$. Let $G=H \cup W(x_1)$ and set $e_m=\{x_1,y_1\}$ be the new whisker at $x_1$ in $G$. Let $G(k_1,...,k_m)$ be constructed as in Construction \ref{con.Gedge}, for some tuple $(k_1, \dots, k_m) \in \NN^m$. Let $\{x_{1,p},y_{1,q}\}$ being the shadows of $\{x_1,y_1\}$, for $p+q\leq k_m+1$. Then, for any $i = 1,...,k_m$, $x_{1,i}$ is a shedding vertex of $G(k_1,...,k_m) \backslash \{x_{1,1},...,x_{1,i-1}\}$.
\end{lemma}

\begin{proof}
By construction, for any $i = 1, \dots, k_m$ we have
$$N_{G(k_1,...,k_m)\backslash\{x_{1,1},...,x_{1,i-1}\}}(y_{1,k_m-(i-1)}) = \{x_{1,j} \mid j+k_m-(i-1)\leq k_m+1\}=\{x_{1,i}\}.$$
This implies that $y_{1,k_m-(i-1)}$ is a simplicial vertex of $G(k_1,...,k_m)\backslash\{x_{1,1},...,x_{1,i-1}\}$. It now follows from \cite[Corollary 7]{Woodroofe2009} that $x_{1,i}$ is a shedding vertex of $G(k_1,...,k_m)\backslash\{x_{1,1},...,x_{1,i-1}\}$.
\end{proof}

\begin{lemma}[\protect{Compare with \cite[Lemma 3.3]{Selvaraja2019}}]
	\label{lem.Sel3}
	Let $G$ be a graph with $m$ edges and let $G(k_1, \dots, k_m)$ be constructed as in Construction \ref{con.Gedge}, for some tuple $(k_1, \dots, k_m) \in \NN^m$. Let $x_i$ be a vertex in $G$ and suppose that $x_i$ is incident to an edge $e_t$, for $1 \le t \le m$.
\begin{enumerate}
	\item If $k_t\leq k$ then
	$$G(k_1,...,k_m)\backslash \{x_{i,1},..,x_{i,k}\} \simeq (G\backslash e_t)(k_1',...,k_{t-1}', 0, k_{t+1}',...,k_m') \cup \{\text{isolated vertices}\}.$$
	\item If $k_t>k$ then
	$$G(k_1,...,k_m)\backslash \{x_{i,1},...,x_{i,k}\}\simeq G(k_1',...,k_{t-1}', k_t-k,k_{t+1}',...,k_m') \cup \{\text{isolated vertices}\}.$$
\end{enumerate}
Here, $k_i' \le k_i$ for all $i = 1, \dots, m$, and $k_l' = \max \{0, k_l - k\}$ for all edges $e_l$ that are adjacent to $x_i$.
\end{lemma}

\begin{proof} Observe first that the deletion of $\{x_{i,1}, \dots, x_{i,k}\}$ from $G(k_1, \dots, k_m)$ only affects the duplication of edges that are adjacent to $x_i$ and shadows of their vertices. Consider such an edge $e_l = \{x_i, x_j\}$.
	
If $k_l \le k$ then removing $\{x_{i,1}, \dots, x_{i,k}\}$ from $G(k_1, \dots, k_m)$ also means removing all shadows of $x_i$ that were introduced in the process of duplicating $e_l$ (by a factor $k_l$). That is, there is no duplication of $e_l$ that survives in $G(k_1, \dots, k_m) \setminus \{x_{i,1}, \dots, x_{i,k}\}$. Particularly, when $k_t \le k$, removing $\{x_{i,1}, \dots, x_{i,k}\}$ from $G(k_1, \dots, k_m)$ gives a duplication of the graph $G \setminus e_t$. Furthermore, the deletion of $\{x_{i,1}, \dots, x_{i,k}\}$ from $G(k_1, \dots, k_m)$ also leaves shadows $x_{j,q}$'s of $x_j$, if $q > k_u$ for any edge $e_u$ adjacent to $x_j$, to be isolated vertices in $G(k_1, \dots, k_m) \setminus \{x_{i,1}, \dots, x_{i,k}\}$.

On the other hand, if $k_l > k$ then removing $\{x_{i,1}, \dots, x_{i,k}\}$ from $G(k_1, \dots, k_m)$ has the same effect as duplicating $e_l$ by a factor $k_l' = k_l -k$ instead of $k_l$, after relabeling $x_{i,k+1}, \dots, x_{i,k_l}$ as $x_{i,1}, \dots, x_{i,k_l-k}$, and possibly adding some isolated shadows of $x_j$ (again, those shadows are $x_{j,q}$ if $q > k_u$ for any edge $e_u$ adjacent to $x_j$). Particularly, the statement when $k_t > k$ follows.
\end{proof}

As a direct consequence of Lemma \ref{lem.Sel3}, we get the following corollary.

\begin{corollary}
	\label{cor.Sel3}
	Let $G$ be a graph with $m$ edges $e_1, \dots, e_m$. Let $x_i$ be a vertex in $G$.
	\begin{enumerate}
		\item If $k_l \le k$ for all edges $e_l$ adjacent to $x_i$ then
		$$G(k_1,...,k_m)\backslash\{x_{i,1},...,x_{i,k}\}\simeq (G\backslash x_i)(k_1',...,k_m') \cup \{\text{isolated vertices}\},$$
		where $k_i'\leq k_i$ for all $i$.
		\item $G(k_1,...,k_m) \backslash N_{G(k_1,...,k_m)}[x_{i,1}] \simeq (G\backslash N_G[x_1])(k_1',...,k_m') \cup \{x_{i,2},...,x_{i,k_\alpha}\}$ where $\alpha = \max\{k_l ~\big|~ e_l \text{ is adjacent to } x_i\}.$
	\end{enumerate}
\end{corollary}

\begin{proof} (1) It follows from Lemma \ref{lem.Sel3} that removing $\{x_{i,1}, \dots, x_{i,k}\}$ from $G(k_1, \dots, k_m)$ will delete all shadows of $x_i$ introduced in duplicating $e_l$ for any edge $e_l$ adjacent to $x_i$. Particularly, all shadows of such $e_l$ are deleted. The statement follows immediately.
	
	(2) Observe that $N_{G(k_1, \dots, k_m)}[x_{i,1}]$ consists of $x_{i,1}$ together with $\{x_{j,1}, \dots, x_{j,k_l}\}$, for all edges $e_l = \{x_i, x_j\}$ adjacent to $x_i$. The statement follows by applying the proof of Lemma \ref{lem.Sel3} repeatedly.
\end{proof}

The following theorem gives a more general statement for our main result in Theorem~\ref{thm.main}. Let us fix some terminology. Let $G$ be graph and let $S$ be a given cycle cover of $G$. We say that a tuple of nonnegative integers $(k_1, \dots, k_m)$, where $m$ is the number of edges in $G \cup W(S)$, satisfies \emph{condition $(\star)$} if for any vertex $x \in S$, the number of shadows of the whisker at $x$ in the graph $(G \cup W(S))(k_1, \dots, k_m)$ is bigger than or equal to that of any edge incident to $x$ in $G$. That is, if $k_i$ represents the number of shadows of the whisker at $x$ and $k_j$ is the number of shadows of any other edge incident to $x$ in $G$, in the construction of $(G \cup W(S))(k_1, \dots, k_m)$, then $k_i \ge k_j$. A particular situation of our interest, where this condition holds, is when $k_1 = \dots = k_m$.

\begin{theorem}
	\label{thm.Gedge}
	Let $G$ be a graph and let $S$ be a cycle cover of $G$. For any tuple of nonnegative integers $(k_1, \dots, k_m)$, where $m$ is the number of edges in $G \cup W(S)$, that satisfies condition $(\star)$, the graph $(G \cup W(S))(k_1, \dots, k_m)$ is vertex decomposable.
\end{theorem}

\begin{proof} We shall use induction on the number $c(G)$ of cycles in $G$. If $c(G) = 0$ then $G$ is a forest and the statement follows from \cite[Theorem 3.3]{KK2020}. Suppose that $c(G) \ge 1$. Set $G' = G \cup W(S)$ and $H = G'(k_1, \dots, k_m)$.
	
We shall apply Lemma \ref{lem.Sel4} to show that $H$ is vertex decomposable. Without loss of generality, we may assume that $x_1\in S$ and $x_1$ belongs to a cycle in $G$. Set $z_i = x_{1,i}$, for $i = 1, \dots, k.$ Let $\Psi_0 = H$, $\Psi_i = \Psi_{i-1} \setminus z_i$ and $\Omega_i = \Psi_{i-1} \setminus N_{\Psi_{i-1}}[z_i]$, for $i = 1, \dots, k$.

It follows from Lemma \ref{lem.Sel2}, $z_i$ is a shedding vertex of $\Psi_{i-1} $. By a re-indexing, if necessary, we may also assume that $e_1, \dots, e_\gamma$, for $\gamma \le \ell$, are the edges on cycles in $G$ which are adjacent to $x_1$ and the whisker at $x_1$. By applying Lemma \ref{lem.Sel3} repeatedly, we have
$$\Psi_k = H\backslash\{z_1,\dots,z_k\}\cong (G'\backslash \{e_1, \dots, e_\gamma\})(k_1',...,k_m') \cup \{\text{isolated vertices}\},$$
where $k_l'\leq k_l$ for all $l = 1, \dots, m$. Observe that $G' \setminus \{e_1, \dots, e_\gamma\} = (G \setminus \{e_1, \dots, e_\gamma\}) \cup \{\text{an isolated vertex}\} \cup W(S \setminus x_1)$ and obviously $c(G\backslash \{e_1, \dots, e_\gamma\}) < c(G)$. Thus, by the induction hypothesis, $\Psi_k$ is vertex decomposable.

In light of Lemma \ref{lem.Sel4}, it suffices to show that $\Omega_i$ is vertex decomposable for all $i = 1, \dots, k$. Indeed, since $N_H(z_j) \subset N_H(z_i)$ for $i\leq j$, we have
$$\Omega_i = H \backslash (\{z_1,...,z_{i-1}\} \cup N_{H}[z_i]) = (G'\backslash x_1)(k_1'',...,k_m'')\cup\{z_{i+1},..,z_k,\mbox{other isolated vertices}\},$$
where $k_l''\leq k_l$ for all $l = 1, \dots, m$. Moreover, $G' \setminus x_1 = (G \setminus x_1) \cup \{\text{an isolated vertex}\} \cup W(S \setminus x_1)$ and clearly $c(G\backslash x_1) < c(G)$. Therefore, by the induction hypothesis, $\Omega_i$ is vertex decomposable for all $i = 1,\dots, k$. The theorem completes.
\end{proof}

As a consequence of Theorem \ref{thm.Gedge}, we obtain a slight generalization of \cite[Theorem 4.3]{KK2020}.

\begin{corollary}
	\label{cor.unicyclic}
		Let $G$ be a unicyclic graph on $m$ edges, with the unique cycle $C = (x_1, \dots, x_{\ell-1})$ and at least a whisker attached to $C$. Suppose that $E_C = \{e_1, \dots, e_{\ell-1}\}$, where $e_j = \{x_j, x_{j+1}\}$ for $j = 1, \dots, \ell-2$ and $e_{\ell-1} = \{x_{\ell-1}, x_1\}$, and a whisker attached to $C$ is $e_{\ell} = \{x_1, x_{\ell}\}$. Then, for any tuple of nonnegative integers $(k_1, \dots, k_m)$ such that $k_\ell \ge k_i$, for all $1 \le i \le m$, the graph $G(k_1, \dots, k_m)$ is vertex decomposable.
\end{corollary}

\begin{proof} It is easy to see that $S = \{x_1\}$ is a cycle cover of $G$. The assertion follows directly from Theorem \ref{thm.Gedge}.
\end{proof}

%\begin{remark}
%	\label{rem.looserCond}
%	The condition $(\star)$ in Theorem \ref{thm.Gedge} and Corollary \ref{cor.unicyclic} can be loosened. In fact, for the same conclusion to hold, it suffices to require that the constant corresponding to the whisker at a vertex $x \in S$ is greater than or equal to the constant corresponding to any edge on a cycle in $G$ that is adjacent to $x$. The same proof works in this case also.
%\end{remark}

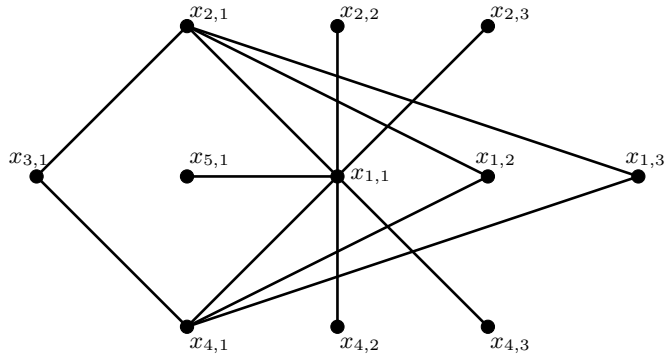
\begin{figure}[h]
	\centering
	\begin{tikzpicture}[line cap=round,line join=round,>=triangle 45,x=1cm,y=1cm]
	\clip(0.06211360697921014,-0.8323164815230916) rectangle (9.604291383294111,4.576982998167132);
	\draw [line width=1pt] (1,2)-- (3,4);
	\draw [line width=1pt] (3,4)-- (5,2);
	\draw [line width=1pt] (5,2)-- (3,0);
	\draw [line width=1pt] (3,0)-- (1,2);
	\draw [line width=1pt] (5,2)-- (3,2);
	\draw [line width=1pt] (5,2)-- (5,4);
	\draw [line width=1pt] (5,2)-- (5,0);
	\draw [line width=1pt] (3,0)-- (7,2);
	\draw [line width=1pt] (7,2)-- (3,4);
	\draw [line width=1pt] (5,2)-- (7,4);
	\draw [line width=1pt] (5,2)-- (7,0);
	\draw [line width=1pt] (3,0)-- (9,2);
	\draw [line width=1pt] (9,2)-- (3,4);
	\begin{scriptsize}
	\draw [fill=black] (1,2) circle (2.5pt);
	\draw[color=black] (0.9017013269082159,2.21312532176184) node {$x_{3,1}$};
	\draw [fill=black] (3,4) circle (2.5pt);
	\draw[color=black] (3.301359190928859,4.152541973701176) node {$x_{2,1}$};
	\draw [fill=black] (5,2) circle (2.5pt);
	\draw[color=black] (5.445251080932077,2.0) node {$x_{1,1}$};
	\draw [fill=black] (3,0) circle (2.5pt);
	\draw[color=black] (3.301359190928859,-0.24849473561150702) node {$x_{4,1}$};
	\draw [fill=black] (3,2) circle (2.5pt);
	\draw[color=black] (3.301359190928859,2.21312532176184) node {$x_{5,1}$};
	\draw [fill=black] (5,4) circle (2.5pt);
	\draw[color=black] (5.301359190928859,4.152541973701176) node {$x_{2,2}$};
	\draw [fill=black] (5,0) circle (2.5pt);
	\draw[color=black] (5.301359190928859,-0.24849473561150702) node {$x_{4,2}$};
	\draw [fill=black] (7,2) circle (2.5pt);
	\draw[color=black] (7.104116881512902,2.21312532176184) node {$x_{1,2}$};
	\draw [fill=black] (7,4) circle (2.5pt);
	\draw[color=black] (7.301359190928859,4.152541973701176) node {$x_{2,3}$};
	\draw [fill=black] (7,0) circle (2.5pt);
	\draw[color=black] (7.301359190928859,-0.24849473561150702) node {$x_{4,3}$};
	\draw [fill=black] (9,2) circle (2.5pt);
	\draw[color=black] (9.099299507469654,2.21312532176184) node {$x_{1,3}$};
	\end{scriptsize}
	\end{tikzpicture}
	\caption{$(G\cup W(S))(3,1,1,3,1)$ is not vertex decomposable.} \label{fig.ex3.10}
\end{figure}

\begin{example}
Condition $(\star)$ on the tuple $(k_1, \dots, k_m)$ in Theorem \ref{thm.Gedge} and Corollary \ref{cor.unicyclic} might be loosened but not removed entirely. Consider $G=C_4$. Then $G\cup W(x_1)$ has a whisker on $x_1$ and meets the cycle cover condition of Theorem \ref{thm.Gedge}. However, letting $e_i = \{x_i, x_{i+1}\}$ for $i=1,2,3$, $e_4 = \{x_1,x_4\}$, and $e_5 =\{x_1,x_5\}$ we can easily verify that $(G\cup W(x_1))(3,1,1,3,1)$, as depicted in Figure \ref{fig.ex3.10}, is not vertex decomposable. This is because $x_{1,1}$ is the only shedding vertex in $(G \cup W(x_1))(3,1,1,3,1)$, and $(G\cup W(x_1))(3,1,1,3,1)\backslash x_{1,1}$ contains no shedding vertices.
\end{example}

Our main result is now a direct consequence of Theorem \ref{thm.Gedge}.

\begin{theorem}
	\label{thm.main}
	Let $G$ be a graph and let $S$ be a cycle cover of $G$. Let $H$ be a graph obtained by adding at least one whisker at each vertex in $S$ and let $J$ be its cover ideal. Then, $J^{(k)}$ is Koszul for all $k \ge 1$.
\end{theorem}

\begin{proof} Observe first that a cycle cover of $G$ remains a cycle cover after adding whiskers to the vertices in $G$. Thus, it suffices to prove the statement for $H = G \cup W(S)$. Let $m$ be the number of edges in $H$. By Theorem \ref{thm.Gedge}, for any $k \ge 1$, $H_k = H(\underbrace{k, \dots, k}_{m \text{ times}})$ is vertex decomposable. This, together with (\ref{eq.polarization}) and Lemma \ref{lem.SF}, implies that $J^{(k)}$ are Koszul for all $k \ge 1$.
\end{proof}

Theorem \ref{thm.main} may not necessarily be true if $S$ is not a cycle cover of $G$. As shown in our next example, the graph in Figure \ref{fig.ex1} gives such an instance.

\begin{example}
	The condition that $S$ be a cycle cover of $G$ is necessary in Theorem \ref{thm.main}. Let $G$ be the ``fish'' graph (a 4-cycle and 3-cycle attached at a vertex) depicted in Figure \ref{fig.ex1}. Let $S=\{x_5\}$. Then $S$ is not a cycle cover of $G$. A quick computation shows that $(G\cup W(S))_2$ is not vertex decomposable; that is $J(G \cup W(S))^{(2)}$ is not Koszul. Note that $G \cup W(S)$ is vertex decomposable in this case.
	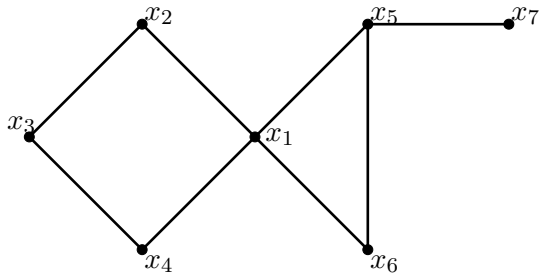
\begin{figure}[h]
		\centering
		\begin{tikzpicture}[scale =.75, line cap=round,line join=round,>=triangle 45,x=1cm,y=1cm]
		\clip(0.49584896044806986,-0.603056366118118) rectangle (10.038026736762967,4.806243113572106);
		\draw [line width=1pt] (1,2)-- (3,4);
		\draw [line width=1pt] (3,4)-- (5,2);
		\draw [line width=1pt] (5,2)-- (3,0);
		\draw [line width=1pt] (3,0)-- (1,2);
		\draw [line width=1pt] (5,2)-- (7,4);
		\draw [line width=1pt] (5,2)-- (7,0);
		\draw [line width=1pt] (7,0)-- (7,4);
		\draw [line width=1pt] (7,4)-- (9.5,4);
		\begin{small}
		\draw [fill=black] (1,2) circle (2.5pt);
		\draw[color=black] (0.8676221205642347,2.213125321761844) node {$x_{3}$};
		\draw [fill=black] (3,4) circle (2.5pt);
		\draw[color=black] (3.301359190928859,4.152541973701176) node {$x_{2}$};
		\draw [fill=black] (5,2) circle (2.5pt);
		\draw[color=black] (5.445251080932077,2.0) node {$x_{1}$};
		\draw [fill=black] (3,0) circle (2.5pt);
		\draw[color=black] (3.301359190928859,-0.24849473561150702) node {$x_{4}$};
		\draw [fill=black] (7,4) circle (2.5pt);
		\draw[color=black] (7.301359190928859,4.152541973701176) node {$x_5$};
		\draw [fill=black] (7,0) circle (2.5pt);
		\draw[color=black] (7.301359190928859,-0.24849473561150702) node {$x_6$};
		\draw [fill=black] (9.5,4) circle (2.5pt);
		\draw[color=black] (9.801359190928859,4.152541973701176) node {$x_7$};
		\end{small}
		\end{tikzpicture}
		\caption{$(G\cup W(S))_2$ is not vertex decomposable when $S$ is not a cycle cover.} \label{fig.ex1}
	\end{figure}
\end{example}

%%%%%%%%%%%%%%%%%%%%%%%%%%%%%%5

\section{Star Complete Graphs and Whiskering} \label{sec.SC}

A whisker can be viewed as a complete graph over two vertices. In this section, we shall see that the conclusion of Theorems \ref{thm.Gedge} and \ref{thm.main} can be made stronger by allowing the addition of larger complete graphs to a vertex rather than just whiskers. This is inspired by the work of Selvaraja \cite{Selvaraja2019}, where pure and non-pure star complete graphs (see Definition \ref{def.SC}) are attached to the vertices of the given graph.

Our method is based on the following construction of ``gluing'' graphs along an edge.

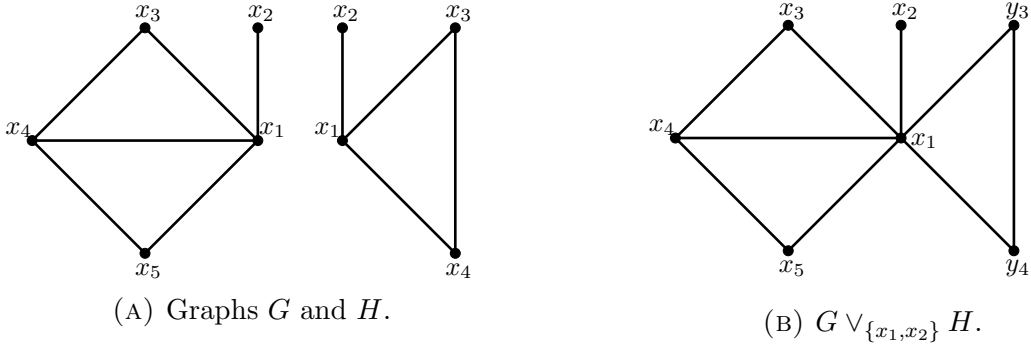
\begin{figure}[h]
	\centering
	\begin{subfigure}[h]{.5\linewidth}
		\centering
		\begin{tikzpicture}[scale =.75, line cap=round,line join=round,>=triangle 45,x=1cm,y=1cm]
		\clip(0.478169148775289,0.5490366834038811) rectangle (9.303249168333453,5.6413152325484095);
		\draw [line width=1pt] (3,1)-- (1,3);
		\draw [line width=1pt] (1,3)-- (3,5);
		\draw [line width=1pt] (3,5)-- (5,3);
		\draw [line width=1pt] (5,3)-- (3,1);
		\draw [line width=1pt] (5,3)-- (1,3);
		\draw [line width=1pt] (5,3)-- (5,5);
		\draw [line width=1pt] (6.5,3)-- (8.5,5);
		\draw [line width=1pt] (8.5,5)-- (8.5,1);
		\draw [line width=1pt] (8.5,1)-- (6.5,3);
		\draw [line width=1pt] (6.5,3)-- (6.5,5);
		\begin{footnotesize}
		\draw [fill=black] (3,1) circle (2.5pt);
		\draw[color=black] (3.0646318960217744,.6894875332764085) node {$x_5$};
		\draw [fill=black] (1,3) circle (2.5pt);
		\draw[color=black] (0.7657397505996578,3.1883796786985254) node {$x_4$};
		\draw [fill=black] (3,5) circle (2.5pt);
		\draw[color=black] (3.0646318960217744,5.253032320216959) node {$x_3$};
		\draw [fill=black] (5,3) circle (2.5pt);
		\draw[color=black] (5.263524041443892,3.1883796786985254) node {$x_1$};
		\draw [fill=black] (5,5) circle (2.5pt);
		\draw[color=black] (5.063524041443892,5.253032320216959) node {$x_2$};
		\draw [fill=black] (6.5,3) circle (2.5pt);
		\draw[color=black] (6.245850362203109,3.1883796786985254) node {$x_1$};
		\draw [fill=black] (8.5,5) circle (2.5pt);
		\draw[color=black] (8.564742507625224,5.253032320216959) node {$x_3$};
		\draw [fill=black] (8.5,1) circle (2.5pt);
		\draw[color=black] (8.564742507625224,.6894875332764085) node {$x_4$};
		\draw [fill=black] (6.5,5) circle (2.5pt);
		\draw[color=black] (6.565850362203109,5.253032320216959) node {$x_2$};
		\end{footnotesize}
		\end{tikzpicture}
		\caption{Graphs $G$ and $H$.}\label{fig:Glue1}
	\end{subfigure}%	
	\begin{subfigure}[h]{.5\linewidth}
		\centering
		\begin{tikzpicture}[scale = .75, line cap=round,line join=round,>=triangle 45,x=1cm,y=1cm]
		\clip(-0.5908972311468136,0.29002257307484725) rectangle (9.57290700315693,5.829693940988442);
		\draw [line width=1pt] (3,1)-- (1,3);
		\draw [line width=1pt] (1,3)-- (3,5);
		\draw [line width=1pt] (3,5)-- (5,3);
		\draw [line width=1pt] (5,3)-- (3,1);
		\draw [line width=1pt] (5,3)-- (1,3);
		\draw [line width=1pt] (5,3)-- (5,5);
		\draw [line width=1pt] (5,3)-- (7,5);
		\draw [line width=1pt] (7,5)-- (7,1);
		\draw [line width=1pt] (7,1)-- (5,3);
		\begin{footnotesize}
		\draw [fill=black] (3,1) circle (2.5pt);
		\draw[color=black] (3.071256774946963,.6894875332764085) node {$x_5$};
		\draw [fill=black] (1,3) circle (2.5pt);
		\draw[color=black] (0.7657397505996578,3.1883796786985254) node {$x_4$};
		\draw [fill=black] (3,5) circle (2.5pt);
		\draw[color=black] (3.071256774946963,5.253032320216959) node {$x_3$};
		\draw [fill=black] (5,5) circle (2.5pt);
		\draw[color=black] (5.0748265391504415,5.253032320216959) node {$x_2$};
		\draw [fill=black] (5,3) circle (2.5pt);
		\draw[color=black] (5.413524041443892,2.950000000000000) node {$x_1$};
		\draw [fill=black] (7,5) circle (2.5pt);
		\draw[color=black] (7.071761966386361,5.253032320216959) node {$y_3$};
		\draw [fill=black] (7,1) circle (2.5pt);
		\draw[color=black] (7.071761966386361,.6894875332764085) node {$y_4$};
		\end{footnotesize}
		\end{tikzpicture}
		\caption{$G\vee_{\{x_1,x_2\}}H$.}\label{fig:Glue2}
	\end{subfigure}
	\caption{The gluing of $G$ and $H$ along $e=\{x_1,x_2\}$.} \label{fig.Glue}
\end{figure}	

\begin{definition}
	Let $G$ and $H$ be graphs over the vertex set $V = \{x_1, \dots, x_n\}$, and let $e = \{x_i, x_j\}$ be a common edge of $G$ and $H$. We construct the \emph{gluing of $G$ and $H$ along $e$}, denoted by $G \vee_e H$, as follows. Let $H'$ be an isomorphic copy of $H$ on the vertex set $V' = \{y_1, \dots, y_n\}$. Then, $G \vee_e H$ is obtained from the disjoint union $G \cupdot H'$ by identifying $x_i$ with $y_i$ and $x_j$ with $y_j$, respectively.
\end{definition}

\begin{example}
	Let $G$ and $H$ be graphs depicted as in Figure \ref{fig:Glue1} and let $e = \{x_1, x_2\}$. Figure \ref{fig:Glue2} illustrates the gluing of $G$ and $H$ along $e$.
\end{example}

Recall that an edge of a graph $G$ is called a \emph{leaf} if one of its endpoints has degree 1. If $e = \{x,y\}$ is a leaf of $G$, where $\deg_G(y) = 1$, then we call $y$ a \emph{leaf vertex} of $G$.

\begin{theorem}
	\label{thm.Glue}
	Let $G$ and $H$ be graphs over the same vertex set with edge sets $\{e, e_2, \dots, e_r\}$ and $\{e, e_2', \dots, e'_s\}$, where $e= \{x_u, x_v\}$ is a common leaf of $G$ and $H$ (with $x_v$ being the leaf vertex). Suppose also that for $(\ell, k_2, \dots, k_r) \in \ZZ^r_{\ge 0}$ and $(\ell, h_2, \dots, h_s) \in \ZZ^s_{\ge 0}$, where $\ell \ge k_i$ and $\ell \ge h_j$ for all $i$ and $j$, the shadows $\{x_{u,1}, \dots, x_{u,\ell}\}$ of $x_u$ in $G(\ell, k_2, \dots, k_r)$ and $H(\ell, h_2, \dots, h_s)$ satisfy the conditions of Lemma \ref{lem.Sel4}. Then, the shadows $\{x_{u,1}, \dots, x_{u,\ell}\}$ of $x_u$ in $(G \vee_e H)(\ell, k_2, \dots, k_r, h_2, \dots, h_s)$ also satisfy the conditions of Lemma \ref{lem.Sel4}. Particularly, $(G \vee_e H)(\ell, k_2, \dots, k_r, h_2, \dots, h_s)$ is vertex decomposable.
\end{theorem}

\begin{proof} In light of Lemma \ref{lem.Sel4}, the later statement follows from the former one. To prove the first statement, let $K = G \vee_e H$ and $A = \{x_{u,1}, \dots, x_{u,\ell}\}$. Set
	$$\Psi_0 = K(\ell, k_2, \dots, k_r, h_2, \dots, h_s), \Psi'_0 = G(\ell, k_2, \dots, k_r) \text{ and } \Psi''_0 = H(\ell, h_2, \dots, h_s).$$
	Furthermore, for $1 \le i \le \ell$, set
\begin{align*}
	& \Psi_i = \Psi_{i-1} \setminus x_{u,i}, \ \Omega_i = \Psi_{i-1} \setminus N_{\Psi_{i-1}}[x_{u,i}], \\
	& \Psi'_i = \Psi'_{i-1} \setminus x_{u,i}, \ \Omega'_i = \Psi'_{i-1} \setminus N_{\Psi'_{i-1}}[x_{u,i}], \\
	& \Psi''_i = \Psi''_{i-1} \setminus x_{u,i}, \ \Omega''_i = \Psi''_{i-1} \setminus N_{\Psi''_{i-1}}[x_{u,i}].
\end{align*}

Observe that, since $\ell \ge k_i$ and $\ell \ge h_j$ for all $i$ and $j$, by Corollary \ref{cor.Sel3} we have $\Psi_\ell = \Psi'_\ell \cupdot \Psi''_\ell$. Thus, $\Psi_\ell$ is vertex decomposable by the hypothesis. Moreover, since $N_\bullet(x_{u,p}) \subseteq N_\bullet(x_{u,q})$ for any $p \ge q$, where $N_\bullet$ can be $N_{K(\ell, k_2, \dots, k_r, h_2, \dots, h_s)}, N_{G(\ell, k_2, \dots, k_r)}$ or $N_{H(\ell, h_2, \dots, h_s)}$, it follows that for $1 \le i \le \ell$, we have
$$\Omega_i = \Omega'_i \cupdot \Omega''_i \cup \{\text{isolated vertices}\}.$$
This, by the hypothesis, implies that $\Omega_i$, for $1 \le i \le \ell$, is vertex decomposable.

It remains to show that $x_{u,i}$ is a shedding vertex of $\Psi_{i-1}$ for all $1 \le i \le \ell$. Indeed, consider any independent set $T$ in $\Omega_i = \Psi_{i-1} \setminus N_{\Psi_{i-1}}[x_{u,i}]$, and set $U = T \cap \Omega'_i$ and $W = T \cap \Omega''_i$. As observed before, $\Omega_i = \Omega'_i \cupdot \Omega''_i$, and so $U$ and $W$ are independent sets in $\Omega'_i$ and $\Omega''_i$, respectively. By the hypothesis, $x_{u,i}$ is a shedding vertex of $\Psi'_{i-1}$. Particularly, $U$ is not a maximal independent set in $\Psi'_i = \Psi'_{i-1} \setminus x_{u,i}$. It follows that there is an independent set $U'$ in $\Psi'_i$ strictly containing $U$. Let $T' = U' \cup W$. Then, $T'$ strictly contains $T$. It can be seen that, since $N_\bullet(x_{u,j}) \subseteq N_\bullet(x_{u,i})$ for all $j \ge i$, no vertex of $\Omega''_i = \Psi''_{i-1} \setminus N_{\Psi''_{i-1}}[x_{u,i}]$ is connected to any vertex in $\Psi'_i$. Therefore, $T'$ is an independent set of $\Psi_i$. Hence, $x_{u,i}$ is a shedding vertex of $\Psi_{i-1}$. The theorem completes.
\end{proof}

We now recall the definition of star complete graphs, following \cite{Selvaraja2019}, and state our result for adding star complete graphs and whiskers to the vertices of a cycle cover, which gives a stronger version of Theorem \ref{thm.Gedge}.

\begin{definition}[\protect{See \cite{Selvaraja2019}}]
	\label{def.SC}
	A \emph{star complete graph} is a graph obtained by joining complete graphs at a common vertex. A star complete graph is called \emph{pure} if all of its maximal cliques contain at least 3 vertices (i.e., it has no whiskers); otherwise, the star complete graph is called \emph{non-pure}.
\end{definition}

\begin{example}
	Figure \ref{fig.SC} gives examples of pure and non-pure star complete graphs. The graph depicted in Figure \ref{fig:Pure} is pure as it all three maximal cliques have 3 or more vertices. On the other hand, the graph depicted in Figure \ref{fig:nonPure} as this graph has two whiskers attached to $x$.
	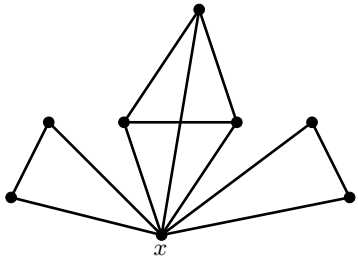
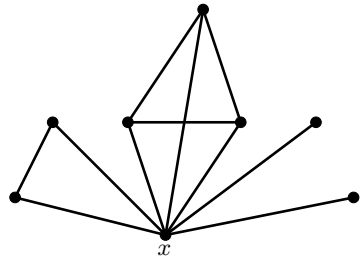
\begin{figure}[h]
		\centering
		\begin{subfigure}[h]{.5\linewidth}
			\centering
			\begin{tikzpicture}[line cap=round,line join=round,>=triangle 45,x=1cm,y=1cm]
			\clip(0.10201142272370023,1.2052607070164092) rectangle (6.559699841073407,4.724953807096036);
			\draw [line width=1pt] (3,1.5)-- (1,2);
			\draw [line width=1pt] (1,2)-- (1.5,3);
			\draw [line width=1pt] (1.5,3)-- (3,1.5);
			\draw [line width=1pt] (3,1.5)-- (2.5,3);
			\draw [line width=1pt] (2.5,3)-- (3.5,4.5);
			\draw [line width=1pt] (3.5,4.5)-- (4,3);
			\draw [line width=1pt] (4,3)-- (3,1.5);
			\draw [line width=1pt] (3,1.5)-- (3.5,4.5);
			\draw [line width=1pt] (2.5,3)-- (4,3);
			\draw [line width=1pt] (3,1.5)-- (5,3);
			\draw [line width=1pt] (5,3)-- (5.5,2);
			\draw [line width=1pt] (5.5,2)-- (3,1.5);
			\begin{scriptsize}
			\draw [fill=black] (3,1.5) circle (2pt);
			\draw[color=black] (2.9852091238667806,1.3060651854053168) node {$x$};
			\draw [fill=black] (1,2) circle (2pt);
			\draw [fill=black] (1.5,3) circle (2pt);
			\draw [fill=black] (2.5,3) circle (2pt);
			\draw [fill=black] (3.5,4.5) circle (2pt);
			\draw [fill=black] (4,3) circle (2pt);
			\draw [fill=black] (5,3) circle (2pt);
			\draw [fill=black] (5.5,2) circle (2pt);
			\end{scriptsize}
			\end{tikzpicture}
			\caption{Pure star complete graph.}\label{fig:Pure}
		\end{subfigure}%	
		\begin{subfigure}[h]{.5\linewidth}
			\centering
			\begin{tikzpicture}[line cap=round,line join=round,>=triangle 45,x=1cm,y=1cm]
			\clip(0.10201142272370023,1.2052607070164092) rectangle (6.559699841073407,4.724953807096036);
			\draw [line width=1pt] (3,1.5)-- (1,2);
			\draw [line width=1pt] (1,2)-- (1.5,3);
			\draw [line width=1pt] (1.5,3)-- (3,1.5);
			\draw [line width=1pt] (3,1.5)-- (2.5,3);
			\draw [line width=1pt] (2.5,3)-- (3.5,4.5);
			\draw [line width=1pt] (3.5,4.5)-- (4,3);
			\draw [line width=1pt] (4,3)-- (3,1.5);
			\draw [line width=1pt] (3,1.5)-- (3.5,4.5);
			\draw [line width=1pt] (2.5,3)-- (4,3);
			\draw [line width=1pt] (3,1.5)-- (5,3);
			\draw [line width=1pt] (5.5,2)-- (3,1.5);
			\begin{scriptsize}
			\draw [fill=black] (3,1.5) circle (2pt);
			\draw[color=black] (2.9852091238667806,1.3060651854053168) node {$x$};
			\draw [fill=black] (1,2) circle (2pt);
			\draw [fill=black] (1.5,3) circle (2pt);
			\draw [fill=black] (2.5,3) circle (2pt);
			\draw [fill=black] (3.5,4.5) circle (2pt);
			\draw [fill=black] (4,3) circle (2pt);
			\draw [fill=black] (5,3) circle (2pt);
			\draw [fill=black] (5.5,2) circle (2pt);
			\end{scriptsize}
			\end{tikzpicture}
			\caption{Non-pure star complete graph.}\label{fig:nonPure}
		\end{subfigure}
	\caption{Pure vs. Non-pure star complete graphs.} \label{fig.SC}
	\end{figure}	
\end{example}

\begin{theorem}
	\label{thm.glueStar}
	Let $G$ be a graph and let $S$ be a cycle cover of $G$. Let $H$ be a graph obtained by attaching a non-pure star complete graph to each vertex in $S$. Then, for all $k \ge 1$, the graph $H_k = H(k, \dots, k)$ is vertex decomposable. Particularly, $J(H)^{(k)}$ is Koszul for all $k \ge 1$.
\end{theorem}

\begin{proof} The second statement follows from the first statement, Lemma \ref{lem.SF} and (\ref{eq.polarization}). We shall now prove that $H_k$ is vertex decomposable for all $k \ge 1$.
	
For each vertex $x \in S$, let $\K(x)$ denote the pure star complete graph that was attached to $x$ in the construction of $H$, excluding all the whiskers at $x$. Set $S' = \{x \in S ~\big|~ \K(x) \not= \emptyset\}$. Let $G'$ be the graph obtained from $G$ by adding the whiskers at the vertices in $S$ in the construction of $H$. In other words, $G'$ is obtained from $H$ by removing $\K(x)$ for $x \in S'$. Clearly, $H = G' \cup \left(\bigcup_{x \in S'} \K(x)\right)$.

By Theorem \ref{thm.Gedge}, $G'_k$ is vertex decomposable. By induction on $|S'|$, to show that $H_k$ is vertex decomposable, it suffices to consider the case when $|S'| = 1$. In this case, let $z$ be the vertex in $S'$ and let $e$ be a whisker at $z$ in $G'$. Set $\K = \K(z) \cup \{e\}$. Then, $H = G' \vee_e \K$.

Observe that, by the proof of \cite[Case 1 of Theorem 3.7]{Selvaraja2019}, the shadows of $z$ in $\K_k$ satisfy the conditions of Lemma \ref{lem.Sel4}. 
On the other hand, by the proof of Theorem \ref{thm.Gedge}, the shadows of $z$ in $(G')_k$ satisfy the conditions of Lemma \ref{lem.Sel4}. Hence, the conclusion now follows from Theorem \ref{thm.Glue} by gluing $G'$ and $\K$ along $e$.
\end{proof}

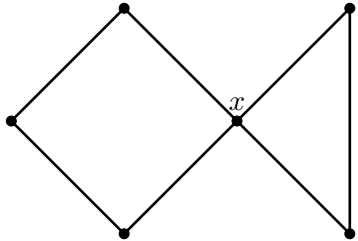
\begin{figure}[h!]
	\centering
	\begin{tikzpicture}[scale = .75, line cap=round,line join=round,>=triangle 45,x=1cm,y=1cm]
	\clip(0.49584896044806986,-0.603056366118118) rectangle (10.038026736762967,4.506243113572106);
	\draw [line width=1pt] (1,2)-- (3,4);
	\draw [line width=1pt] (3,4)-- (5,2);
	\draw [line width=1pt] (5,2)-- (3,0);
	\draw [line width=1pt] (3,0)-- (1,2);
	\draw [line width=1pt] (5,2)-- (7,4);
	\draw [line width=1pt] (5,2)-- (7,0);
	\draw [line width=1pt] (7,0)-- (7,4);
	\begin{small}
	\draw [fill=black] (1,2) circle (2.5pt);
	%\draw[color=black] (0.8676221205642347,2.213125321761844) node {$x_{3}$};
	\draw [fill=black] (3,4) circle (2.5pt);
	%\draw[color=black] (3.301359190928859,4.152541973701176) node {$x_{2}$};
	\draw [fill=black] (5,2) circle (2.5pt);
	\draw[color=black] (5,2.3) node {$x$};
	\draw [fill=black] (3,0) circle (2.5pt);
	%\draw[color=black] (3.301359190928859,-0.24849473561150702) node {$x_{4}$};
	\draw [fill=black] (7,4) circle (2.5pt);
	%\draw[color=black] (7.301359190928859,4.152541973701176) node {$x_5$};
	\draw [fill=black] (7,0) circle (2.5pt);
	%\draw[color=black] (7.301359190928859,-0.24849473561150702) node {$x_6$};
	\end{small}
	\end{tikzpicture}
	\caption{Attaching pure star complete graphs to a cycle cover.} \label{fig.nonPureNec}
\end{figure}

\begin{example}
	The condition that non-pure star complete graphs are attached to the vertices of a cycle cover in Theorem \ref{thm.glueStar} cannot be removed. Consider a 4-cycle $C_4$ and one of its vertices, say $x$. Clearly, $S = \{x\}$ is a cycle cover of $C_4$. Let $G$ be the graph obtained by attaching a triangle to $x$; see Figure \ref{fig.nonPureNec}. Then, $G$ is the ``fish'' graph as described in Figure \ref{fig.ex1}. Computation shows that $G_2$ is not vertex decomposable.
\end{example}

We end the paper with the following questions.

\begin{question}
	Is the conclusion of Theorem \ref{thm.glueStar} still true if we also add pure star complete graphs to the vertices in $G \setminus S$?
\end{question}

\begin{question}
	Find a necessary and sufficient condition on a subset $S$ of the vertices in a graph $G$ such that $J(G \cup W(S))^{(k)}$ is Koszul for all $k \ge 1$.
\end{question}

%%%%%%%%%%%%%%%%%%%%%%%%%%%%%%%%%%%%%%%%%%%%%%%%%%%
\bibliographystyle{abbrv}
\bibliography{Reference}

\end{document}